\documentclass[a4paper,12pt,twoside]{article}
\usepackage[latin2]{inputenc}
\usepackage{ipe}
\usepackage[mathscr]{eucal}
\usepackage[dvips]{graphicx}
\usepackage{amssymb,amsmath,amsthm}
\usepackage{amsfonts}
\usepackage{paralist}
\usepackage{ifthen}
\usepackage{url}
\makeatletter

\def\Fraisse{Fra\"{\i}ss\' e}
\def\K{{\cal K}}

\def\TV{{\cal TV}}
\def\Rado{{\cal R}}
\def\Words{{\cal W}}
\def\Functions{{\cal F}}
\def\Convex{{\cal C}}
\def\Intervals{{\cal I}}
\def\Bintree{{\cal B}}
\def\Periodic{{\cal S}}
\def\Grammar{{\cal G}}
\def\Convex{{\cal C}}

\def\Z{\mathbb Z}
\def\Poset{\mathbb P}

\DeclareMathSymbol{\Oddel}
   {\mathbin}{symbols}{"6A}

\def\He{{\cal P_\in}}

\def\Model{{\mathfrak {M}}}

\def\Z{{\mathbb Z}}

\def\Polynoms{{\cal O}}

\def\Poset{{\cal P}}

\def\Paths{{\cal P}}

\def\apple{{\heartsuit}}
\def\HE#1,#2,#3{(\,#2\,\Oddel\,#3\,)}
\def\Embed#1#2{\Phi^{#1}_{#2}}
\def\TA{\HE 1,\emptyset,\emptyset}
\def\TB{\HE 2,\emptyset,{\{\TA\}}}
\def\TC{\HE 3,{\{\TA,\TB\}},\emptyset}
\def\TD{\HE 4,{\{\TB\}},{\{\TC\}}}
\def\Sur{\mathbb S}

\newcommand{\subf}[1][]{\ifthenelse{\equal{#1}{}}{\ensuremath\leq}{\ensuremath\leq_{#1}}}
\newcommand{\psubf}[1][]{\ifthenelse{\equal{#1}{}}{\ensuremath <}{\ensuremath <_{#1}}}
\newcommand{\supf}[1][]{\ifthenelse{\equal{#1}{}}{\ensuremath\succeq}{\ensuremath\succeq_{#1}}}
\newcommand{\psupf}[1][]{\ifthenelse{\equal{#1}{}}{\ensuremath\succ}{\ensuremath\succ_{#1}}}
\newcommand{\incomp}[1][]{\ifthenelse{\equal{#1}{}}{\ensuremath\parallel}{\ensuremath\parallel_{#1}}}
\newcommand{\fequiv}[1][]{\ifthenelse{\equal{#1}{}}{\ensuremath\equiv}{\ensuremath\equiv_{#1}}}
\newcommand{\subc}[1][]{\ifthenelse{\equal{#1}{}}{\ensuremath\preccurlyeq}{\ensuremath\preccurlyeq_{#1}}}
\newcommand{\eqclass}[2][]{\ifthenelse{\equal{#1}{}}{\ensuremath[{#2}]}{\ensuremath[{#2}]_{#1}}}

\newtheorem{defn}{Definition}[section]
\newtheorem{fact}{Fact}

\newtheorem{corollary}{Corollary}[section]
\newtheorem{prop}{Proposition}[section]
\newtheorem{remark}{Remark}[section]

\newtheorem{example}{Example}[section]

\newtheorem{thm}{Theorem}[section]
\newtheorem{lem}[thm]{Lemma} 

\DeclareRobustCommand{\qed}{%
  \ifmmode \mathqed
  \else
    \leavevmode\unskip\penalty9999 \hbox{}\nobreak\hfill
    \quad\hbox{$\square$}%
  \fi
}

\begin{document}

\title{Some examples of universal and generic partial orders}

\author{
   \normalsize {\bf Jan Hubi\v cka$^*$}\\
   \normalsize {\bf Jaroslav Ne\v set\v ril}\thanks {The Institute for Theoretical Computer Science (ITI) is supported as project 1M0545 by the Ministry of Education of the Czech Republic.}
\\
   {\small Department of Applied Mathematics}\\
{\small and}\\
{\small Institute of Theoretical Computer sciences (ITI)}\\
   {\small Charles University}\\
   {\small Malostransk\' e n\' am. 25, 11800 Praha, Czech Republic}\\
   {\normalsize 118 00 Praha 1}\\
   {\normalsize Czech Republic}\\
   {\small \{hubicka,nesetril\}@kam.ms.mff.cuni.cz}
}

\date{}
\maketitle
\begin{abstract}
We survey structures endowed with natural partial orderings and prove their
universality.  These partial orders include partial orders on sets of words,
partial orders formed by geometric objects, grammars, polynomials and homomorphism order for various combinatorial objects.
\end{abstract}
\section {Introduction}

For given class $\K$ of countable partial orders we say that class $\K$
contains an {\em embedding-universal} (or simply {\em universal}) structure
$(U,\leq_U)$ if every partial order $(P,\leq_P)\in \K$ can be found as induced
suborder of $(U,\leq_U)$ (or in other words, there exists embedding from
$(P,\leq_P)$ to $(U,\leq_U)$).

A partial order $(P,\leq_P)$ is {\em ultrahomogeneous} (or simply {\em homogeneous}),
if every isomorphism of finite suborders of $(P,\leq_P)$ can be extended to an
automorphism of $(P,\leq_P)$. 

A partial order $(P,\leq_P)$ is {\em generic} if it is both {\em ultrahomogeneous} and {\em universal}.

The generic objects can be obtained from the \Fraisse{} limit \cite{F}. But it
is important that often these generic objects (despite their apparent
complexity and universality) admit a concise presentation. Thus for example the
Rado graph (i.e. countable universal and homogeneous undirected graph) can be
represented in various ways by elementary properties of sets or finite
sequences, number theory or even probability. Similar concise representations
were found for some other generic objects such as all undirected
ultrahomogeneous graphs \cite{HN-Posets} or the Urysohn space
\cite{metric}.  The study of generic partial order also motivated
this paper and we consider representation of the generic partial order in Section 3.

The notion of finite presentation we interpret here broadly as a succinct
representation of an infinite set, succint in the sense that elements are
finite models with relations induced by ``compatible mappings'' (such as
homomorphisms) between the corresponding models. This intuitive definition
suffices as we are interested in the (positive) examples of such
representations.

A finite presentation of the generic partial order is given in \cite{HN-Posets}
---however this construction is quite complicated. This paper gives a more
streamlined construction and relates it to Conway surreal numbers (see Section
3). 

In Section 2 we present several simple constructions which yield (countably) universal
partial orders. Such objects are interesting on their own and were intensively studied in
the context of universal algebra and categories. For example, it is a classical
result of Pultr and Trnkov\'a \cite{Pultr} that finite graphs with the
homomorphism order are countably universal quasiorder. Extending and completing
\cite{HN-Posets} we give here several constructions which yields to universal
partial orders.  These constructions include: 
\begin{enumerate}
\item The order $(\Words,\leq_\Words)$ on sets of words in the alphabet $\{0,1\}$.
\item The dominance order on the binary tree $(\Bintree,\leq_\Bintree)$.
\item The inclusion order of finite sets of finite intervals $(\Intervals,\leq_\Intervals)$.
\item The inclusion order of convex hulls of finite sets of points in the plane $(\Convex,\leq_\Convex)$.
\item The order of piecewise linear functions on rationals $(\Functions,\leq_\Functions)$.
\item The inclusion order of periodic sets $(\Periodic,\subseteq)$.
\item The order of sets of truncated vectors (generalization of orders of vectors
of finite dimension) $(\TV,\leq_\TV)$.
\item The orders implied by grammars on words $(\Grammar,\leq_\Grammar)$.
\item The homomorphism order of oriented paths $(\Paths,\leq_\Paths)$.
\end{enumerate}

Note that with universal partial orders we have more freedom (than with the generic
pratial order) and as a consequence we give a perhaps surprising variety of finite
presentations.

We start with a simple representation by means of finite sets of binary words. This
representation seems to capture properties of such a universal partial order very
well and it will serve as our ``master'' example.  In most other cases we prove the
universality of some particular partial order by finding a mapping from the
words representation into the structure in question.  This technique will be
shown in several applications in the next sections. While some of these
structures are known be universal, see e.g. \cite{hedrlin,nesu,HN-paths}, in several
cases we can prove the universality in a new, we believe, much easier way.
The embeddings of structures are presented as follows (ones denoted by dotted
lines are not presented in this paper, but references are given).

\medskip
~

\centerline{\includegraphics{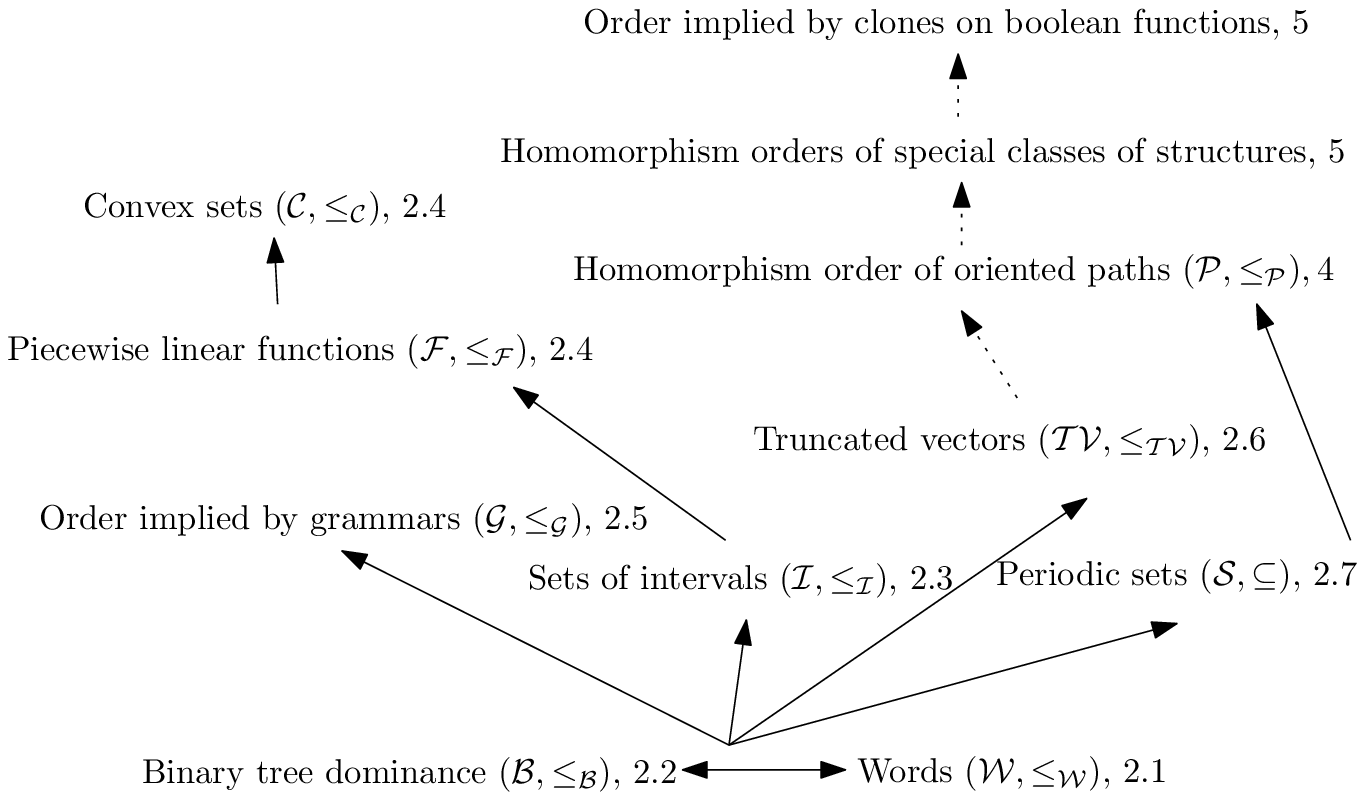}}

\medskip
~

At this point we would like to mention that the (countable) universality
is an essentially finite problem as it can be formulated as follows:
By an {\em on-line representation} of a class $\K$ of partial orders in a partial
order $(P,\leq_P)$, we mean that one can construct an embedding $\varphi:R\to P$
of any partial order $(R,\leq_R)$ in class $\K$ under the restriction that
the elements of $R$ are revealed one by one.  The on-line representation of a
class of partial orders can be considered as a game between two players $A$ and
$B$ (usually Alice and Bob).  Player $B$ chooses a partial order $(P,\leq_P)$ in the class $\K$, and
reveals the elements of $P$ one by one to player $A$ ($B$ is a bad guy).
Whenever an element of $x$ of $P$ is revealed to $A$, the relations among $x$
and previously revealed elements are also revealed.  Player $A$ is required to
assign a vertex $\varphi(x)$---before the next element is revealed---such that
$\varphi$ is an embedding of the suborder induced by  $(P,\leq_P)$ on the already revealed
elements of $(R,\leq_R)$.  Player $A$ wins the game if he succeeds in
constructing an embedding $\varphi$.  The class $\K$ of partial orders is
on-line representable in the partial order $(P,\leq_P)$ if player $A$ has a winning
strategy. 

On-line representation (describing winning strategy of $A$) is a convenient way
of showing the universality of given partial order. In particular it transforms
problem of embedding countable structures into a finite problem of extending the
existing partial embedding by next element.

We say that a partial order $(P,\leq_P)$ has the {\em extension property} if the
following holds: for any finite mutually disjoint subsets $L,G,U\subseteq P$ there exist a vertex
$v\in P$ such that $v'<_P v$ for each $v'\in L$, $v<_P v'$ for each $v'\in G$
and neither $v\leq_P v'$ nor $v'\leq_P v$ for each $v'\in U$.  The extension
property is a stronger form of on-line representability of any partial order.
Using a zig-zag argument it is easy to show that a partial order having the extension
property is homogeneous (and thus generic).

In Section 3 we describe a finite representation of the generic partial order related
to Conway's surreal numbers.  Somewhat surprisingly, this is the only known finite presentation of the generic partial order \cite{HN-Posets}. The
constructions of universal partial orders are easier, but they are often not
generic. We discuss reasons why other structures fail to be ultrahomogeneous.  In
particular we will look for gaps in the partial order.  Recall that the {\em
gap} in a partial order $(P,\leq_P)$ is a pair of elements $v, v'\in P$ such that
$v<_\Bintree v'$.  A partial order having no gaps is called {\em dense}. We will
show examples of universal partial orders both with gaps  and without gaps but
still failing to be generic.

\section{Examples of Universal Partial Orders}

To prove the universality of a given partially ordered set 
is often a difficult task \cite{hedrlin,Pultr,HN-trees,nesu}.
However, the individual proofs, even if developed independently, use similar tools.
We demonstrate this by isolating a ``master'' construction (in Section \ref{wordsection}). This construction is then embedded into partial orders defined
by other structures (as listed above). We shall see that the representation of
this particular order is flexible enough to simplify further embeddings.

\subsection{Word representation}
\label{wordsection}

The set of all words over the alphabet $\Sigma=\{0,1\}$ is denoted by $\{0,1\}^*$.
For words $W,W'$ we write $W\leq_w W'$ if and only if $W'$ is an initial segment (left
factor) of $W$. Thus we have, for example, $\{011000\}\leq_w \{011\}$ and
$\{010111\}\nleq_w \{011\}$.

\begin{defn}
\label{Wordsdef}

Denote by $\Words$ the class of all finite subsets $A$ of $\{0,1\}^*$ such that no distinct words $W,W'$ in $A$ satisfy $W\leq_w W'$. For
$A,B\in \Words$ we put $A\leq_\Words B$ when for each $W\in A$ there exists
$W'\in B$ such that $W\leq_w W'$.
\end{defn}

Obviously $(\Words,\leq_\Words)$ is a partial order (antisymmetry follows from the fact that $A$ is an antichain in the order $\leq_w$).

\begin{defn}
For a set $A$ of finite words denote by $\min A$ the set of all minimal words
in $A$ (i.e. all $W\in A$ such that there is no $W'\in A$ satisfying $W'<_wW$).
\end{defn}

Now we show that there is an on-line embedding of any finite partial order to $(\Words,\leq_\Words)$.
Let $[n]$ be the set $\{1,2,\ldots, n\}$.  The partial orders will be restricted to those whose vertex sets are sets $[n]$ (for some $n>1$) and
the vertices will always be embedded in the natural order.  Given a partial order
$([n],\leq_P)$ let $([i],\leq_{P_i})$ denote the partial order induced by
$([n],\leq_P)$ on the set of vertices $[i]$. 

Our main construction is the function $\Psi$ mapping partial orders $([n],\leq_P)$ to elements of $(\Words,\leq_\Words)$ defined as follows:
\begin{defn}
\label{algu}
~

Let $L([n],\leq_P)$ be the union of all $\Psi([m],\leq_{P_m})$, $m<n$, $m\leq_P n$.

Let $U([n],\leq_P)$ be the set
of all words $W$ such that $W$ has length $n$, the last letter is $0$ and
for each $m<n,n\leq_P m$ there is a $W'\in \Psi([m],\leq_{P_m})$ such that $W$ is an initial
segment of $W'$.

Finally, let  $\Psi([n],\leq_P)$ be $\min (L([n],\leq_P) \cup U([n],\leq_P))$.  

In particular, $L([1],\leq_P)=\emptyset, U([1],\leq_P)=\{0\},\Psi ([1],\leq_P)=\{0\}$.
\end{defn}

The main result of this section is the following:

\begin{thm}
\label{Pembed}
Given a partial order $([n],\leq_P)$ we have:
\begin{enumerate}
\item For every $i,j\in[n]$, $$i\leq_P j \hbox{ if and only if } \Psi([i],\leq_{P_i})\leq_\Words \Psi([j],\leq_{P_j})$$ and $$\Psi([i],\leq_{P_i}) = \Psi([j],\leq_{P_j}) \hbox{ if and only if } i=j.$$ (This says that the mapping $\Phi(i) = \Psi([i],\leq_{P_i})$ is an embedding of $([n],\leq_P)$ into $(\Words,\leq_\Words)$),
\item for every $S\subseteq [n]$ there is a word $W$ of length $n$ such that for each $k\leq n$, $\{W\}\leq_\Words\Psi([k],\leq_{P_k})$ if and only if either $k\in S$ or there is a $k'\in S$ such that $k'\leq_P k$.
\end{enumerate}
\end{thm}
The on-line embedding $\Phi$ is illustrated by the following example:
  \begin{figure}[t]
    \begin{center}
      {\def\IPEfile{poset.ipe}\begingroup
  \catcode`\%=9\catcode`\!=0\catcode`\-=11\input{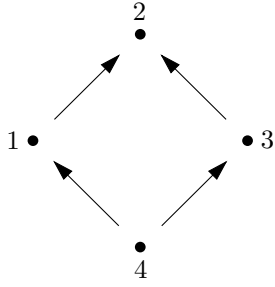}}
    \end{center}
    \caption{The partial order $([4],\leq_P)$.}
    \label{poset}
  \end{figure}
\begin{example}
The partial order $([4],\leq_P)$ depicted in Figure~\ref{poset} has the following values of $\Psi([k],\leq_{P_k}), k=1,2,3,4$:

$$
\begin{array}{lll}
L([1],\leq_{P_1})=\emptyset,& U([1],\leq_{P_1})=\{0\},& \Psi([1],\leq_{P_1})=\{0\},\\
L([2],\leq_{P_2})=\{0\},& U([2],\leq_{P_2})=\{00,10\},& \Psi([2],\leq_{P_2})=\{0,10\},\\
L([3],\leq_{P_3})=\emptyset,& U([3],\leq_{P_3})=\{000,100\},& \Psi([3],\leq_{P_3})=\{000,100\},\\
L([4],\leq_{P_4})=\emptyset,& U([4],\leq_{P_4})=\{0000\},& \Psi([4],\leq_{P_4})=\{0000\}.\\
\end{array}
$$
\end{example}
\begin{proof}[Proof (of Theorem \ref{Pembed})]
We proceed by induction on $n$.

The theorem obviously holds for $n=1$. 

Now assume that the theorem holds for every partial order $([i],\leq_{P_i})$, $i=1,\ldots, n-1$.

We first show that $2.$ holds for $([n],\leq_P)$.  Fix $S\subseteq\{1,2,\ldots, n\}$.
Without loss of generality assume that for each $m\leq n$ such that there is an
$m'\in S$ with $m'\leq_P m$, we also have $m\in S$ (i.e. $S$ is closed upwards).  By the induction hypothesis, there is a word $W$ of length $n-1$ such that for each $n'<n$, $\{W\}\leq_\Words
\Psi([n'],\leq_{P_{n'}})$ if and only if $n'\in S$. Given the word $W$ we can construct a word $W'$ of length $n$ such that $\{W'\}\leq_\Words \Psi([n'],\leq_{P_{n'}})$ if and only if $n'\in S$. To see this, consider the following cases: 
\begin{enumerate}
\item $n\in S$
\begin{enumerate}
\item $\{W\}\leq_\Words \Psi([n],\leq_P)$. Put $W'=W0$. Since $\{W'\}\leq_\Words\{W\}$, $W'$ obviously has the property.

\item $\{W\}\nleq_\Words \Psi([n],\leq_P)$.  In this case we have $m\in S$ for each $m<n,n\leq_P m$, and thus $\{W\}\leq_\Words \Psi([m],\leq_{P_m})$.  By the definition of $\leq_\Words$, for each such $m$ we have $W''\in \Psi([m],\leq_{P_m})$ such that $W''$ is an initial segment of $W$. This implies that $W0$ is in $U([n],\leq_P)$ and thus $\{W\}\leq_{\Words}\Psi([n],\leq_P)$, a contradiction.
\end{enumerate}
\item $n\notin S$
\begin{enumerate}
\item $\{W\}\nleq_\Words \Psi([n],\leq_P)$. In this case we can put either $W'=W0$ or $W'=W1$.
\item $\{W\}\leq_\Words \Psi([n],\leq_P)$.  We have $\{W\}\nleq_{\Words} L([n],\leq_P)$---otherwise we
would have $\{W\}\leq_{\Words} \Psi([m],\leq_{P_m})\leq_{\Words} \Psi([n],\leq_P)$ for some $m<n$ and thus $n\in
S$. Since $U([n],\leq_P)$ contains words of length $n$ whose last digit is 0 putting $W'=W1$ gives $\{W'\}\nleq_\Words U([n],\leq_P)$ and thus also $\{W'\}\nleq_\Words \Psi([m],\leq_{P_m})$.
\end{enumerate}
\end{enumerate}

\bigskip
\noindent
This finishes the proof of property $2.$

Now we prove $1.$  We only need to verify that for $m=1,2,\ldots, n-1$ we have $\Psi([n],\leq_P${}$)\leq_\Words \Psi([m],\leq_{P_m})$ if and only if $n\leq_P m$ and $\Psi([m],\leq_{P_m})\leq_\Words \Psi([n],\leq_P)$ if and only if $m\leq_P n$.  The rest follows by induction. Fix $m$ and consider the following cases:

\begin{enumerate}
\item
{\em $m\leq_P n$ implies $\Psi([m],\leq_{P_m})\leq_\Words \Psi([n],\leq_P)$}: This follows easily from the fact that every word in $\Psi([m],\leq_{P_m})$ is in $L([n],\leq_P)$ and the initial segment of each word in $L([n],\leq_P)$ is in $\Psi([n],\leq_P)$).

\item
{\em $n\leq_P m$ implies $\Psi([n],\leq_P)\leq_\Words \Psi([m],\leq_{P_m})$}: $U([n],\leq_P)$ is a maximal set of words of length $n$ with last digit $0$ such that $U([n],\leq_P)\leq_\Words \Psi([m'],\leq_{P_{m'}})$ for each $m'<n,n\leq_P m'$, in particular for $m'=m$. It suffices to show that $L([n],\leq_P)\leq_\Words \Psi([m],\leq_{P_m})$. For $W\in L([n],\leq_P)$, we have an $m''$, $m''\leq_P n\leq_P m$, such that $W\in \Psi([m''],\leq_{P_{m''}})$.  From the induction hypothesis $ \Psi([m''],\leq_{P_{m''}})\leq_\Words \Psi([m],\leq_{P_m})$---in particular the initial segment of $W$ is in $\Psi([m],\leq_{P_m})$.

\item
{\em $\Psi([m],\leq_{P_m})\leq_\Words \Psi([n],\leq_P)$ implies $m\leq_P n$}:
Since $U([n],\leq_P)$ contains words longer than any word of $m$, we have $\Psi([m],\leq_{P_m})\leq_\Words L([n],\leq_P)$.  By $2.$ for $S=\{m\}$ there is a word $W$ such that $\{W\}\leq_\Words  \Psi([m'],\leq_{P_{m'}})$ if and only if $m\leq_P m'$.  Since $\{W\}\leq_\Words L([n],\leq_P)$, we have an $m'$ such that $m\leq_P m'\leq_P n$.

\item
{\em $\Psi([n],\leq_P)\leq_\Words \Psi([m],\leq_{P_m})$ implies $n\leq_P m$}: We have $\Psi([n],\leq_P)\leq_\Words \Psi([m],\leq_{P_m})$.   By $2.$ for $S=\{n\}$ there is a word $W$ such that $\{W\}\leq_\Words  \Psi([m'],\leq_{P_{m'}})$ if and only if $n\leq_P m'$. Since $\{W\}\leq_\Words \Psi([m],\leq_{P_m})$ we also have $n\leq_P m$.
\end{enumerate}

\end{proof}
\begin{corollary}
The partial order $(\Words,\leq_\Words)$ is universal.
\end{corollary}

Note that $\Words$ fails to be a ultrahomogeneous partial order. For example the empty set is the minimal element. $\Words$ is also not dense as shown by the following example: $$A=\{0\}, B=\{00,01\}.$$

This is not unique gap---we shall characterize all gaps in $(\Words,\leq_\Words)$ after
reformulating it in a more combinatorial setting in Section \ref{dominance}.

\subsection{Dominance in the countable binary tree}
\label{dominance}
As is well known, the Hasse diagram of the partial order $(\{0,1\}^*, \leq_w)$
forms a complete binary tree $T_u$ of infinite depth. Let $r$ be its root
vertex (corresponding to the empty word).  Using $T_u$ we can reformulate our
universal partial order as:

\begin{defn}
The vertices of $(\Bintree,\leq_\Bintree)$ are finite sets $S$ of vertices of $T_u$ such that there is no
vertex $v\in S$ on any path from $r$ to $v'\in S$ except for $v'$. (Thus $S$ is a finite antichain in the order of the tree $T$.)

We say that $S'\leq_\Bintree S$ if and only if for each path from $r$ to $v\in S$ there is a vertex
$v'\in S'$.
\end{defn}

\begin{corollary}
The partially ordered set $(\Bintree,\leq_\Bintree)$ is universal.
\end{corollary}
\begin{proof}
$(\Bintree,\leq_\Bintree)$ is just a reformulation of $(\Words,\leq_\Words)$ and thus both partial orders are isomorphic.
\end{proof}

  \begin{figure}[t]
    \begin{center}
      \includegraphics[width=10cm]{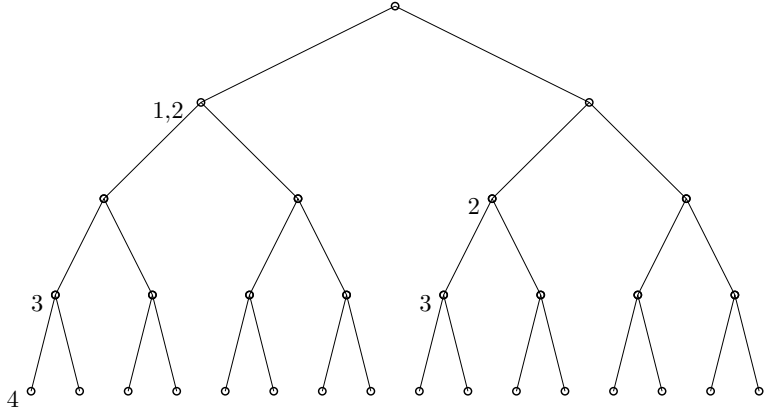}
    \end{center}
    \caption{Tree representation of $([4],\leq_P)$ (Figure~\ref{poset}).}
    \label{strom-plny}
  \end{figure}
Figure~\ref{strom-plny} shows a portion of the tree $T$ representing
the same partial order as in Figure~\ref{poset}.

The partial order $(\Bintree,\leq_\Bintree)$ offers perhaps a better intuitive
understanding as to how the universal partial order is built from the very simple partial
order $(\{0,1\}^*, \leq_w)$ by using sets of elements instead of single
elements.  Understanding this makes it easy to find an embedding
of $(\Words,\leq_\Words)$ (or equivalently $(\Bintree,\leq_\Bintree)$) into a new structure by first looking for a way to represent the partial order $(\{0,1\}^*, \leq_w)$ within the new structure and then a way to represent subsets of $\{0,1\}^*$. This idea will be applied several times in the following sections.

Now we characterize gaps.

\begin{prop}
$S<S'$ is a gap in $(\Bintree,\leq_\Bintree)$ if and only if there exists an $s'\in S'$ such that
\begin{enumerate}
\item there is a vertex $s\in S$ such that both sons $s_0,s_1$ of $s$ in the tree $T$ are in $S'$,
\item $S\setminus\{s_0,s_1\}=S'\setminus\{s\}$.
\end{enumerate}
\end{prop}
This means that all gaps in $\Bintree$ result from replacing a member by its two sons.
\begin{proof}
Clearly any pair $S<S'$ satisfying $1.$, $2.$ is a gap (as any $S\leq_\Bintree S''\leq_\Bintree S'$ has to contain $S'\setminus\{s\}$, and either $s$ or the two vertices $s_0$, $s_1$).  

Let $S\leq_\Bintree S'$ be a gap.  If there are distinct vertices $s'_1$ and $s'_2$ in $S'$ and $s_1,s_2\in S$ are such that $s_i\leq s'_i$, i=1,2, then $S''$ defined as $min(S\setminus\{s_1\})\cup\{S'_1\}$ satisfies $S<_\Bintree S''<_\Bintree S'$.

Thus there is only one $S'\in S'\setminus S$ such that $s'>s$ for an $s\in S$. However then there is only one such $s'$ (so if $s_1,s_2$ are distinct then $S<S\setminus\{s_2\}<S'$). Moreover it is either $s=s'0$ or $s=s'1$. Otherwise $S<S'$ would not be a gap.
\end{proof}

The abundance of gaps indicates that $(\Bintree,\leq_\Bintree)$ (or
$(\Words,\leq_\Words)$) are redundant universal partial orders. This makes
them, in a way, far from being generic, since the generic partial order has no
gaps.  The next section has a variant of this partial order avoiding this
problem.  On the other hand gaps in partial orders are interesting and are
related to dualities, see \cite{Trotter, NZhu}.

\subsection{Intervals}
We show that the vertices of $(\Words,\leq_\Words)$ can be coded by geometric
objects ordered by inclusion.  Since we consider only countable structures we restrict ourselves to objects formed from rational numbers.

While the interval on rationals ordered by inclusion can represent infinite increasing chains,
decreasing chains or antichains, obviously this interval order has 
dimension 2 and thus fails to be universal. However
considering multiple intervals overcomes this limitation:
\begin{defn}
The vertices of $(\Intervals,\leq_\Intervals)$ are finite sets $S$ of closed
disjoint intervals  $[a,b]$ where $a$, $b$ are rational numbers and
$0\leq a<b\leq 1$.

We put $A\leq_\Intervals B$ when every interval in $A$ is covered by some interval of $B$. 
\end{defn}

In the other words elements of $(\Intervals,\leq_\Intervals)$ are finite sets of pairs of rational
numbers.  $A\leq_\Intervals B$ holds if for every $[a,b]\in A$, there is
an $[a',b']\in B$ such that $a'\leq a$ and $b\leq b'$.

\begin{defn} 
 A word $W=w_1w_2\ldots w_t$ on the alphabet $\{0,1\}$ can be considered as a number $0\leq n_W\leq 1$ 
with ternary expansion: $$n_W=\sum_{i=1}^{t} w_i{1\over 3^i}.$$  

For $A\in \Words$, the representation of $A$ in $\Intervals$ is then the following set of
intervals:

$$\Embed{\Words}{\Intervals}(A)=\{[n_W,n_W+{2\over 3^{|W|+1}}]: W\in A\}.$$
\end{defn}

The use of the ternary base might seem unnatural---indeed the binary base would
suffice. The main obstacle to using the later is that the embedding of $\{00,01\}$ would be two
intervals adjacent to each other overlapping in single point.  One would need
to take special care when taking the union of such intervals---we avoid this by
using ternary numbers.

\begin{lem}
$\Embed{\Words}{\Intervals}$ is a embedding of $(\Words,\leq_{\Words})$ into $(\Intervals,\leq_{\Intervals})$.

\end{lem}
\begin{proof}
It is sufficient to prove that for $W$, $W'$ there is an interval $[n_W,n_W+{1\over
3^{|W|}}]$ covered by an interval $[n_{W'},n_{W'}+{1\over 3^{|W'|}}]$ if and only if $W'$ is
initial segment of $W$. This follows easily from the fact that
intervals represent precisely all numbers whose ternary expansion starts with
$W$ with the exception of the upper bound itself.
\end{proof}

\begin{example}
The representation of $([4],\leq_P)$ as defined by Figure~\ref{poset} in $(\Intervals,\leq_\Intervals)$ is:
$$
\begin{array}{lllll}
\Embed{\Words}{\Intervals}(\Psi([1],\leq_{P_1}))&=&\Embed{\Words}{\Intervals}(\{0\})&=&\{(0,{2\over 3^2})\},\\
\Embed{\Words}{\Intervals}(\Psi([2],\leq_{P_2}))&=&\Embed{\Words}{\Intervals}(\{0,10\})&=&\{(0,{2\over 3^2}),({1\over 3}, {1\over 3}+{2\over 3^3})\},\\
\Embed{\Words}{\Intervals}(\Psi([3],\leq_{P_3}))&=&\Embed{\Words}{\Intervals}(\{000,100\})&=&\{(0,{2\over 3^4}),({1\over 3},{1\over 3}+{2\over 3^4})\},\\
\Embed{\Words}{\Intervals}(\Psi([4],\leq_{P_4}))&=&\Embed{\Words}{\Intervals}(\{0000\})&=&\{(0,{2\over 3^5})\}.
\end{array}
$$
\end{example}

\begin{corollary}
The partial order $(\Intervals,\leq_\Intervals)$ is universal.
\end{corollary}

The partial order $(\Intervals,\leq_\Intervals)$ differs significantly from
$(\Words,\leq_\Words)$ by the following:

\begin{prop}
The partial order $(\Intervals,\leq_\Intervals)$ has no gaps (is dense).
\end{prop}
\begin{proof}
Take $A,B\in \Intervals$, $A<_\Intervals B$. Because all the intervals in both $A$ and $B$ are closed and disjoint, there must be at least one interval $I$ in $B$ that is not fully covered by intervals of $A$ (otherwise we would have $B\leq_\Intervals A$). We may construct an element $C$ from $B$ by shortening the interval $I$  or splitting it into two disjoint intervals in a way such that $A<_\Intervals C<_\Intervals B$ holds.
\end{proof}
Consequently the presence (and abundance) of gaps in most of the universal
partial orders studied is not the main obstacle when looking for
representations of partial orders.  It is easy to see that
$(\Intervals,\leq_\Intervals)$ is not generic.

By considering a variant of $(\Intervals,\leq_\Intervals)$ with open 
(instead of closed) intervals we obtain a universal partial order
$(\Intervals',\leq_{\Intervals'})$ with gaps. The gaps are similar to the ones in
$(\Bintree,\leq_\Bintree)$ created by replacing interval $(a,b)$ by two
intervals $(a,c)$ and $(c,d)$.  Half open intervals give a
quasi-order containing a universal partial order.

\subsection{Geometric representations}
The representation as a set of intervals might be considered an
artificially constructed structure.  Partial orders represented by geometric
objects are studied in \cite{Alon}. It is shown that objects with $n$ ``degrees
of freedom'' cannot represent all partial orders of dimension $n+1$. It follows that
convex hulls used in the representation of the generic partial order cannot be defined by a constant
number of vertices.  We will show that even the simplest geometric objects with
unlimited ``degrees of freedom'' can represent a universal partial order.

\begin{defn}
Denote by $(\Convex,\leq_\Convex)$ the partial order whose vertices are all convex hulls
of finite sets of points in $\mathbb Q^2$, ordered by inclusion.
\end{defn}
This time we will embed $(\Intervals,\leq_\Intervals)$ into $(\Convex,\leq_\Convex)$.
\begin{defn}
For every $A\in \Intervals$ denote by $\Embed{\Intervals}{\Convex}(A)$ the convex hull generated by the points: $$(a,a^2),
({{a+b}\over 2}, ab), (b,b^2)\hbox{, for every } (a,b)\in A.$$
\end{defn}
  \begin{figure}[t]
    \begin{center}
1:
      \includegraphics[width=5cm]{t.1}
2:
      \includegraphics[width=5cm]{t.2}\\
3:
      \includegraphics[width=5cm]{t.3}
4:
      \includegraphics[width=5cm]{t.4}
    \end{center}
    \caption{Representation of the partial order $([4],\leq_P)$ in $(\Convex,\leq_\Convex)$.}
    \label{posetconvex}
  \end{figure}
See Figure~\ref{posetconvex} for the representation of the partial order in Figure~\ref{poset}.
\begin{thm}
$\Embed{\Intervals}{\Convex}$ is an embedding of $(\Intervals,\leq_\Intervals)$ to $(\Convex,\leq_\Convex)$.
\end{thm}
\begin{proof}
All points of the form $(x,x^2)$ lie on a convex parabola $y=x^2$. The points $
({{a+b}\over 2}, ab)$ are the intersection of two tangents of this parabola at the
points $(a,a^2)$ and $(b,b^2)$. Consequently all points in the construction of
$\Embed{\Intervals}{\Convex}(A)$ lie in a convex configuration.

We have $(x,x^2)$ in the convex hull $\Embed{\Intervals}{\Convex}(A)$ if and only if there is $[a,b]\in A$ such that $a\leq x\leq b$.  Thus for $A,B\in \Intervals$ we have $\Embed{\Intervals}{\Convex}(A)\leq_\Convex\Embed{\Intervals}{\Convex}(B)$ implies $A\leq_\Intervals B$.

To see the other implication, observe that the convex hull of $(a,a^2)$, $({{a+b}\over 2}, ab)$, $(b,b^2)$ is a subset of the convex hull of $(a',a'^2)$, $({{a'+b'}\over 2}, a'b'),$ $(b',b'^2)$ for every $[a,b]$ that is a  subinterval of $[a',b']$.
\end{proof}

We have:
\begin{corollary}
The partial order $(\Convex,\leq_\Convex)$ is universal.
\end{corollary}

\begin{remark}
Our construction is related to Venn diagrams.  Consider the partial order $([n],\leq_P)$.  For the empty relation $\leq_P$ the representation constructed by $\Embed{\Intervals}{\Convex}(\Embed{\Words}{\Intervals}(\Psi([n],\emptyset)))$ is a Venn diagram, by Theorem \ref{Pembed} $(2.)$. Statement 2 of Theorem \ref{Pembed} can be seen as a Venn diagram condition under the constraints imposed by $\leq_P$.
\end{remark}

The same construction can be applied to functions, and stated in a perhaps more precise manner.
\begin{corollary}
Consider the class $\Functions$ of all convex piecewise linear functions  on the interval $(0,1)$ consisting of a finite set of segments, each with rational boundaries.  Put $f\leq_\Functions g$ if and only if $f(x)\leq g(x)$ for every $0\leq x\leq 1$. Then the partial order $(\Functions,\leq_\Functions$) is universal.
\end{corollary}

Similarly the following holds:
\begin{thm}
Denote by $\Polynoms$ the class of all finite polynomials with rational coefficients.  For $p,q\in \Polynoms$, put $p\leq_\Polynoms q$ if and only if $p(x)\leq q(x)$ for $x\in (0,1)$. The partial order
$(\Polynoms,\leq_\Polynoms)$ is universal.
\end{thm}
The proof of this theorem needs tools of mathematical analysis
The proof of this theorem needs more involved tools of mathematical analysis
and it will appear elsewhere (jointly with Robert \v S\'amal).

\subsection{Grammars}

The rewriting rules used in a context-free grammar can be also used to define
a universal partially ordered set.  
\begin{defn}

The vertices of $(\Grammar,\leq_\Grammar)$ are all words over the alphabet $\{\downarrow{},\uparrow{},0,1\}$ created from the word $1$ by the following rules:
$$
\begin{array}{rcl}
1&\to& \downarrow{}11\uparrow{},\\
1&\to& 0.
\end{array}
$$

$W\leq_\Grammar W'$ if and only if $W$ can be constructed from $W'$ by:
$$
\begin{array}{rcl}
1&\to& \downarrow{}11\uparrow{},\\
1&\to& 0,\\
\downarrow{}00\uparrow{}&\to& 0.
\end{array}$$
\end{defn}

$(\Grammar,\leq_\Grammar)$ is a quasi-order: the transitivity of $\leq_\Grammar$ follows from the composition of lexical transformations.

\begin{defn}
Given $A\in \Words$ construct $\Embed{\Words}{\Grammar}$ as follows:

\begin{enumerate}
\item $\Embed{\Words}{\Grammar}(\emptyset)=0$.

\item $\Embed{\Words}{\Grammar}(\{\hbox{empty word}\})=1$.

\item $\Embed{\Words}{\Grammar}(A)$ is defined as the concatenation $\downarrow\Embed{\Words}{\Grammar}(A_0)\Embed{\Words}{\Grammar}(A_1)\uparrow$, where $A_0$ is created from all words of $A$ starting with $0$ with the first digit  removed and $A_1$ is created from all words of $A$ starting with $1$ with the first digit removed.
\end{enumerate}
\end{defn}

\begin{example}
The representation of $([4],\leq_P)$ as defined by Figure~\ref{poset} in $(\Grammar,\leq_\Grammar)$ is as follows (see also the correspondence with the $\Bintree$ representation in Figure~\ref{strom-plny}):
$$
\begin{array}{lllll}
\Embed{\Words}{\Grammar}(\Psi([1],\leq_{P_1}))&=&\Embed{\Words}{\Grammar}(\{0\})&=&\downarrow{}10\uparrow{},\\
\Embed{\Words}{\Grammar}(\Psi([2],\leq_{P_2}))&=&\Embed{\Words}{\Grammar}(\{0,10\})&=&\downarrow{}1\downarrow{}10\uparrow{}\uparrow{},\\
\Embed{\Words}{\Grammar}(\Psi([3],\leq_{P_3}))&=&\Embed{\Words}{\Grammar}(\{000,100\})&=&\downarrow{}\downarrow{}\downarrow{}10\uparrow{}0\uparrow{}\downarrow{}\downarrow{}10\uparrow{}0\uparrow{}\uparrow{},\\
\Embed{\Words}{\Grammar}(\Psi([4],\leq_{P_4}))&=&\Embed{\Words}{\Grammar}(\{0000\})&=&\downarrow{}\downarrow{}\downarrow{}\downarrow{}10\uparrow{}0\uparrow{}0\uparrow{}0\uparrow{}.\\
\end{array}
$$
\end{example}
We state the following without proof as it follows straightforwardly from the definitions.
\begin{prop}
For $A,B\in \Words$ the inequality $A\leq_\Words B$ holds if and only if $\Embed{\Words}{\Grammar}(A)\leq_\Grammar \Embed{\Words}{\Grammar} (B)$.
\end{prop}
$(\Grammar,\leq_\Grammar)$ is a quasi-order. We have:
\begin{corollary}
The quasi-order $(\Grammar,\leq_\Grammar)$ contains a universal partial order.
\end{corollary}
\subsection{Multicuts and Truncated Vectors}
\label{TVsection}
A universal partially ordered structure similar to $(\Words,\leq_\Words)$, but
less suitable for further embeddings, was studied in
\cite{hedrlin,nesu,HN-trees}. While the structures defined in these papers are easily shown to be
equivalent, their definition and motivations were different.  \cite{hedrlin} contains the first finite presentation of universal partial order. \cite{nesu} first
used the notion of on-line embeddings to (1) prove the universality of the structure and
(2) as intermediate structure to prove the universality of the homomorphism
order of multigraphs. The motivation for this structure came from the  analogy with
Dedekind cuts and thus its members were called {\em multicuts}.  In
\cite{HN-trees} an essentially equivalent structure with the inequality reversed was
used as an intermediate structure for the stronger result showing the universality of
oriented paths. This time the structure arises in the context of orders of vectors
(as the simple extension of the orders of finite dimension represented by finite
vectors of rationals) resulting in name {\em truncated vectors}.

We follow the presentation in \cite{HN-trees}.

\begin{defn}
\label{TVdef}
Let $\vec{v}=(v_1,\ldots,v_t)$, $\vec{v}'=(v'_1,\ldots,v'_{t'})$ be $0$--$1$ vectors.
We put: $$\vec{v}\leq_{\vec{v}}\vec{v}'\hbox{ if and only if }t\geq t'\hbox{ and }v_i\geq v'_i\hbox{ for }i=1,\ldots,t'.$$  
\end{defn}
Thus we have e.g. $(1,0,1,1,1)<_{\vec{v}}(1,0,0,1)$ and $(1,0,0,1)>_{\vec{v}}(1,1,1,1)$.
An example of an infinite descending chain is e.g.
$$(1)>_{\vec{v}}(1,1)>_{\vec{v}}(1,1,1)>_{\vec{v}}\ldots.$$
Any finite partially ordered set is representable by vectors with this
ordering: for vectors of a fixed length we have just the reverse ordering of that used in
the (Dushnik-Miller) dimension of partially ordered sets, see e.g.
\cite{Trotter}.

\begin{defn}
We denote by $\TV$ the class of all finite vector-sets.
Let $\vec{V}$ and $\vec{V}'$ be two finite sets of $0$--$1$ vectors.  
We put $\vec{V}\leq_\TV\vec{V}'$ if and only if for every $\vec{v}\in \vec{V}$ there exists a $\vec{v}'\in\vec{V}'$ such that $\vec{v}\leq_{\vec{v}}\vec{v}'$.
\end{defn}

For a word $W$ on the alphabet $\{0,1\}$ we construct a vector $\vec{v}(W)$
of length $2|W|$ such that $2n$-th element of vector $\vec{v}(W)$ is $0$ if and only if the
$n$-th character of $W$ is 0, and the $(2n+1)$-th element of the vector $\vec{v}(W)$ is $1$
if and only if the $n$-th character of $W$ is 0.

It is easy to see that $W\leq_\Words W'$ if and only if
$\vec{v}(W)\leq_{\vec{v}} \vec{v}(W')$.
The embedding $\Embed{\Words}{\TV}:(\Words,\leq_\Words)\to (\TV,\leq_\TV)$ is
constructed as follows:
$$\Embed{\Words}{\TV}= \{\vec{v}(W), W\in A\}.$$

For our example $([4],\leq_P)$ in Figure~\ref{poset} we have embedding:

$$
\begin{array}{lllll}
\Embed{\Words}{\TV}(\Psi([1],\leq_{P_1}))&=&\Embed{\Words}{\TV}(\{0\})&=&\{(0,1)\},\\
\Embed{\Words}{\TV}(\Psi([2],\leq_{P_2}))&=&\Embed{\Words}{\TV}(\{0,10\})&=&\{(0,1),(1,0,0,1)\},\\
\Embed{\Words}{\TV}(\Psi([3],\leq_{P_3}))&=&\Embed{\Words}{\TV}(\{000,100\})&=&\{(0,1,0,1,0,1),(1,0,0,1,0,1)\},\\
\Embed{\Words}{\TV}(\Psi([4],\leq_{P_4}))&=&\Embed{\Words}{\TV}(\{0000\})&=&\{(0,1,0,1,0,1,0,1)\}.\\
\end{array}
$$

\begin{corollary}
The quasi-order $(\TV,\leq_\TV)$ contains a universal partial order.
\end{corollary}

The structure $(\TV,\leq_\TV)$ as compared to $(\Words,\leq_\Words)$ is more complicated to use for further
embeddings: the partial order of vectors is already a complex finite-universal
partial order. The reason why the structure $(\TV,\leq_\TV)$ was discovered first is is that it
allows a remarkably simple on-line embedding that we outline now.

Again we restrict ourselves to the partial orders whose vertex sets are the sets
$[n]$ (for some $n>1$) and we will always embed the vertices in the natural order.
The function $\Psi'$ mapping partial orders $([n],\leq_P)$ to elements of
$(\TV,\leq_\TV)$ is defined as follows:

\begin{defn}
Let $\vec{v}({[n],\leq_{P}})=(v_1$, $v_2$, \ldots, $v_{n})$ where $v_m=1$ if and only if $n\leq_\Poset m$, $m\leq n$, otherwise $v_m=0$. 

Let $$\Psi'([n],\leq_P)=\{\vec{v}([m],\leq_{P_m}): m\in P, m\leq n, m\leq_\Poset n\}.$$
\end{defn}

For our example in Figure~\ref{poset} we get a different (and more compact) embedding:
$$
\begin{array}{llllll}
\vec{v}(1)&=&(1), &\Psi'([1],\leq_{P_1})&=&\{(1)\},\\
\vec{v}(2)&=&(0,1), &\Psi'([2],\leq_{P_2})&=&\{(1),(0,1)\},\\
\vec{v}(3)&=&(1,0,1), &\Psi'([3],\leq_{P_3})&=&\{(1,0,1)\},\\
\vec{v}(4)&=&(1,1,1,1), &\Psi'([4],\leq_{P})&=&\{(1,1,1,1)\}.
\end{array}
$$

\begin{thm}
Fix the partial order $([n],\leq_P)$.
 For every $i,j\in[n]$, $$i\leq_P j \hbox{ if and only if } \Psi'([i],\leq_{P_i})\leq_\TV \Psi'([j],\leq_{P_j})$$ and $$\Psi'([i],\leq_{P_i}) = \Psi'([j],\leq_{P_j}) \hbox{ if and only if } i=j.$$ (Or in the other words, the mapping $\Phi'(i) = \Psi'([i],\leq_{P_i})$ is the embedding of $([n],\leq_P)$ into $(\TV,\leq_\TV)$).
\end{thm}

The proof can  be done via induction analogously as in the second part of the proof of
Theorem \ref{Pembed}. See our paper \cite{HN-trees}.  The main advantage of this embedding is
that the size of the answer is $O(n^2)$ instead of $O(2^n)$. 

\subsection{Periodic sets}
\label{preiodickesection}
As the last finite presentation we mention the following what we believe to
be very elegant description.
Consider the partial order defined by inclusion on sets of integers. This
partial order is uncountable and contains every countable partial order. We
can however show the perhaps surprising fact that the subset of all periodic subsets (which has a very simple and finite description) is countably universal.
\begin{defn}
$S\subseteq \Z$ is {\em $p$-periodic} if for every $x\in S$ we have also $x+p\in S$ and $x-p\in S$.

For a periodic set $S$ with period $p$ denote by the {\em signature $s(p,S)$} a word over the alphabet $\{0,1\}$ of length $p$ such that $n$-th letter is $1$ if and only if $n\in S$.

By $\Periodic$ we denote the class of all sets $S\subseteq \Z$ such that $S$ is
$2^n$-periodic for some $n$.
\end{defn}

Clearly every periodic set is determined by its signature and thus
$(\Periodic,\subseteq)$ is a finite presentation. We consider the ordering
of periodic sets by inclusion and prove:

\begin{thm}
The partial order $(\Periodic,\subseteq)$ is universal.
\end{thm}

\begin{proof}
We embed $(\Words,\leq_\Words)$ into $(\Periodic,\subseteq)$ as follows:
For $A\in \Words$ denote by $\Embed\Words{\Periodic}(A)$ the set of integers such that $n\in \Embed\Words{\Periodic}(A)$ if and only if there is $W\in A$ and the $|A|$ least significant digits of the binary expansion of $n$ forms a reversed word $W$ (when the binary expansion has fewer than $|W|$ digits, add 0 as needed).

It is easy to see that $\Embed\Words{\Periodic}(A)$ is $2^n$-periodic, where $n$ is the length of
longest word in $W$, and $\Embed\Words{\Periodic}(A)\subseteq \Embed\Words{\Periodic}(A')$ if and only if $A\leq_\Words A'$.
\end{proof}

$(\Periodic,\subseteq)$ is dense, but it fails to have the $3$-extension
property: there is no set strictly smaller than the set with signature $01$ and
greater than both sets with signatures $0100$ and $0010$.

\section{Generic Poset and Conway numbers}
One of the striking (and concise) incarnations of the generic Rado graph is provided by the set theory: vertices of $\Rado$ are all sets in a fixed countable model $\Model$ of the theory of finite sets, and the edges correspond to pairs $\{A,B\}$ for which either $A\in B$ or $B\in A$. In \cite{HN-Posets} we aimed for a similarly concise representation of a generic partial order. That appeared to be a difficult task and we had to settle for the weaker notion of ``finite presentation''. 
At present \cite{HN-Posets} is the only finite presentation of the generic partial order.
This is related to Conway surreal numbers \cite{conway,Conway}.

  In this section, for completeness we give the finite presentation of the generic partial order as shown in \cite{HN-Posets}.  This construction is of
independent interest as one can give a finite presentation of the rational
Urysohn space along the same lines \cite{metric}.
We work in a fixed
countable model $\Model$ of the theory of finite sets extended by a single atomic
set $\apple$.  To represent ordered pairs $(M_L,M_R)$, we use following notation:
  $$M_L=\{A;A\in M,\apple\notin A\};$$
  $$M_R=\{A;(A\cup{\{\apple\}})\in M,\apple\notin A\}.$$

\begin{defn}
  \label{Hdef}
  Define the partially ordered set $(\He,\leq_\in)$ as follows:
  
  The elements of $\He$ are all sets $M$
  with the following properties:
   \begin{enumerate}
     \item(correctness)
	\begin{enumerate}
	\item $\apple\notin M$,
	\item $M_L\cup M_R\subset \He$,
	\item $M_L\cap M_R = \emptyset$.
	\end{enumerate}
     \item(ordering property) $(\{A\}\cup A_R)\cap(\{B\}\cup B_L)\neq \emptyset$ for each $A\in M_L, B\in M_R$,
     \item(left completeness) $A_L\subseteq M_L$ for each $A\in M_L$,
     \item(right completeness) $B_R\subseteq M_R$ for each $B\in M_R$.
   \end{enumerate}
  The relation of $\He$ is denoted by $\leq_\in$ and it is defined as follows:  We put $M<_\in N$ if
  $$(\{M\}\cup M_R)\cap (\{N\}\cup N_L)\neq \emptyset.$$
  We write $M\leq_\in  N$ if either $M<_\in N$ or $M=N$.
\end{defn}

The class $\He$ is non-empty  (as  $M=\emptyset=
\HE 0,\emptyset,\emptyset \in \He$). (Obviously the correctness property
holds. Since $M_L=\emptyset$, $M_R=\emptyset$, the ordering property and
completeness properties follow trivially.)

Here are a few examples of non-empty elements of the structure $\He$:
$$
  \begin{array}{c}
    \TA,\\
    \TB,\\
    \TC.\\
  \end{array}
$$

It is a non-trivial fact that $(\He,\leq_\in)$ is a partially ordered set.  This will be proved after
introducing some auxiliary notions:

\begin{defn}

  Any element $W\in (A\cup A_R)\cap(B\cup B_L)$ is called a {\em witness}
of the inequality $A<_\in B$.
\end{defn}
\begin{defn}

  The {\em level of $A\in \He$} is defined as follows:
\begin{eqnarray}
l(\emptyset)&=&0,\nonumber \\
l(A)&=&max(l(B):B\in A_L\cup A_R) + 1 \hbox{ for } A\neq \emptyset.\nonumber
\end{eqnarray}
\end{defn}

We observe the following facts (which follow directly from the
definition of $\He$):
\begin{fact}
 \label{lpmnoziny}
 $X <_\in A <_\in Y$ for every $A\in \He$, $X\in A_L$ and $Y\in A_R$.
\end{fact}
\begin{fact}
\label{fact2}
$A\leq_\in W^{AB}\leq_\in B$ for any $A<_\in B$ and witness $W^{AB}$ of $A<_\in B$.
\end{fact}
\begin{fact}
\label{level}
Let $A<_\in B$ and let $W^{AB}$ be a witness of $A<_\in B$. Then
$l(W^{AB})\leq \min(l(A), l(B))$, and either $l(W^{AB})<l(A)$ or $l(W^{AB})<l(B)$.
\end{fact}
First we prove transitivity.
  \begin{lem}
    \label{ltr}
    The relation $\leq_\in$ is transitive on the class $\He$.
  \end{lem}
  \begin{proof}
   Assume that three elements $A,B,C$ of $\He$ satisfy $A<_\in B<_\in C$. We prove that $A<_\in C$
   holds.  Let $W^{AB}$ and $W^{BC}$ be witnesses of the inequalities $A<_\in B$
   and $B<_\in C$ respectively. First we prove that $W^{AB}\leq_\in W^{BC}$.  We distinguish four cases (depending on the definition of the witness):
   \begin{enumerate}
      \item $W^{AB}\in B_{L}$ and $W^{BC}\in B_{R}$.
      
      In this case it follows from Fact \ref{lpmnoziny} that $W^{AB}<_\in W^{BC}$.
      \item $W^{AB}=B$ and $W^{BC}\in B_{R}$. 
      
      Then $W^{BC}$ is a witness of the inequality $B<_\in W^{BC}$ and thus $W^{AB}<_\in W^{BC}$.
      \item $W^{AB}\in B_{L}$ and $W^{BC} = B$.
      
      The inequality $W^{AB}\leq_\in W^{BC}$ follows analogously to the previous case.
      \item $W^{AB}=W^{BC}=B$ (and thus $W^{AB}\leq_\in W^{BC}$).
   \end{enumerate}
   In the last case  $B$ is a witness of the inequality $A<_\in C$.
   Thus we may assume that $W^{AB}\neq_\in W^{BC}$.  Let $W^{AC}$ be a witness of
   the inequality $W^{AB}<_\in W^{BC}$.  Finally we prove that $W^{AC}$ is a witness
   of the inequality $A<_\in C$. We distinguish three possibilities:
   \begin{enumerate}
      \item $W^{AC}=W^{AB}=A$. 
      \item $W^{AC}=W^{AB}$ and $W^{AC}\in A_R$.  
      \item $W^{AC}\in W^{AB}_R$, then also $W^{AC}\in A_R$ from the completeness property.
   \end{enumerate}
   It follows that either $W^{AC}=A$ or $W^{AC}\in A_R$.  Analogously either
   $W^{AC}=C$ or $W^{AC}\in C_L$ and thus $W^{AC}$ is the witness of inequality
   $A<_\in C$.
  \end{proof}
  \begin{lem}
    \label{lasy}
    The relation $<_\in$ is strongly antisymmetric on the class $\He$.
  \end{lem}
  \begin{proof}
    Assume that $A<_\in B<_\in A$ is a counterexample with minimal $l(A)+l(B)$.
    Let $W^{AB}$ be a witness of the inequality $A<_\in B$ and $W^{BA}$ a witness of
    the reverse inequality.  From Fact \ref{fact2} it follows that $A\leq_\in W^{AB}\leq_\in B\leq_\in W^{BA}\leq_\in A \leq_\in W^{AB}$.
    From the transitivity we know that $W^{AB}\leq_\in W^{BA}$ and $W^{BA}\leq_\in W^{AB}$.

    Again  we  consider 4 possible cases:
    \begin{enumerate}
      \item $W^{AB}=W^{BA}$.
      
      From the disjointness of the sets $A_L$ and $A_R$ it follows that $W^{AB}=W^{BA}=A$.  Analogously we obtain
      $W^{AB}=W^{BA}=B$, which is a contradiction.
      \item Either $W^{AB}=A$ and $W^{BA}=B$ or $W^{AB}=B$ and $W^{BA}=A$. 
      
      Then
      a contradiction follows in both cases from the fact that $l(A)<l(B)$ and $l(B)<l(A)$ (by Fact \ref{level}).
      \item $W^{AB}\neq A$, $W^{AB}\neq B$, $W^{AB}\neq W^{BA}$.  
      
      Then $l(W^{AB})<l(A)$ and
      $l(W^{AB})< l(B)$. Additionally  we have $l(W^{BA})\leq l(A)$ and $l(W^{BA})\leq
      l(B)$ and thus $A$ and $B$ is not a minimal counter example.
      \item $W^{BA}\neq A$, $W^{BA}\neq B$, $W^{AB}\neq W^{BA}$.
      
      The contradiction follows symmetrically to the previous case from the minimality of $l(A)+l(B)$.
    \end{enumerate}
  \end{proof}
\begin{thm}\label{order}
  $(\He,\leq_\in)$ is a partially ordered set.
\end{thm}
  \begin{proof}
    Reflexivity of the relation follow directly from the definition, transitivity and antisymmetry follow from  Lemmas \ref{ltr} and \ref{lasy}.
  \end{proof}

Now we are ready to prove the main result of this section:

\begin{thm}
  \label{univ2}
  $(\He,\leq_\in)$ is the generic partially ordered set for the class of all countable partial orders.
\end{thm}
First we show the following lemma:
\begin{lem}
  \label{ext}
  $(\He,\leq_\in)$ has the extension property.
\end{lem}
\begin{proof}
  Let $M$ be a finite subset of the elements of $\He$.  We want to extend the
  partially ordered set induced by $M$ by the new element
  $X$.  This extension can be described by three subsets of $M$: $M_-$ containing
  elements smaller than $X$, $M_+$ containing elements greater than $X$, and $M_0$
  containing elements incomparable with $X$.  Since the extended relation is
  a partial order we have the following properties of these sets:
  \begin{enumerate}
    \item[I.] Any element of $M_-$ is strictly smaller than any element of $M_+$,
    \item[II.] $B\leq_\in A$ for no $A\in M_-$, $B\in M_0$,
    \item[III.] $A\leq_\in B$ for no $A\in M_+$, $B\in M_0$,
    \item[IV.] $M_-$, $M_+$ and $M_0$ form a partition of $M$.
  \end{enumerate}

  Put
  $$\overline{M_-}=\bigcup_{B\in M_-}B_L\cup M_-,$$
  $$\overline{M_+}=\bigcup_{B\in M_+}B_R\cup M_+.$$
  We verify that the properties I., II., III., IV. still hold for sets $\overline{M_-}$, $\overline{M_+}$, $M_0$.

  \begin{enumerate}
    \item [ad I.]
    We prove that any element of $\overline{M_-}$ is strictly smaller than any element of $\overline{M_+}$:
    
    Let $A\in\overline {M_-}, A'\in\overline {M_+}$.  We prove $A<_\in A'$.
    By the definition of $\overline {M_-}$ there exists $B\in M_-$ such that either $A=B$ or $A\in B_L$.
    By the definition of $\overline {M_+}$ there exists $B'\in M_+$ such that either $A'=B'$ or $A'\in B'_R$.
    By the definition of $<_\in $ we have $A\leq_\in B$, $B<_\in B'$ (by I.) and $B'\leq_\in A'$ again by the definition of $<_\in $.  It follows $A<_\in A'$.
    \item [ad II.]
    We prove that $B\leq_\in A$ for no $A\in \overline{M_-}$, $B\in M_0$:

    Let $A\in\overline {M_-}, B\in M_0$ and let
    $A'\in M_-$ satisfy either $A=A'$ or $A\in A'_L$.
    We know that $B\nleq_\in A'$ and as $A\leq_\in A'$ we have also $B\nleq_\in A$.
    \item [ad III.] To prove that $A\leq_\in B$ for no $A\in \overline{M_+}$, $B\in M_0$ we can proceed similarly to ad II.
    \item [ad IV.]
    We prove that $\overline{M_-}$, $\overline{M_+}$ and $M_0$ are pairwise disjoint:
    
    $\overline {M_-}\cap \overline {M_+}=\emptyset$ follows from I.
    $\overline {M_-}\cap M_0=\emptyset$ follows from II.
    $\overline {M_+}\cap M_0=\emptyset$ follows from III.
  \end{enumerate}

  It follows that $A=\HE N,{\overline {M_-}},{\overline {M_+}}$ is an element
  of $\He$ with the desired inequalities for the elements in the sets $M_-$ and $M_+$.

  Obviously each element of $M_-$ is smaller than $A$ and each element of $M_+$ is greater than $A$.

  It remains to be shown that each $N\in M_0$ is incomparable with $A$. However
  we run into a problem here: it is possible that $A=N$.
  We can avoid this problem by first considering the set:
 $$M'=\bigcup_{B\in M}B_R\cup M.$$
 It is then easy to show that $B=\HE 0,\emptyset,{M'}$ is an element of $\He$
 strictly smaller than all elements of $M$.

 Finally we construct the set $A'=\HE 0,{A_L\cup\{B\}},{A_R}$.  The set $A'$ has the same
 properties with respect to the elements of the sets $M_-$ and $M_+$ and differs from any set in
 $M_0$.  It remains to be shown that $A'$ is incomparable with $N$.

 For contrary, assume for example, that $N<_\in A'$ and $W^{NA'}$ is the witness of the
 inequality.  Then $W^{NA'}\in  \overline {M_-}$ and $N\leq_\in W^{NA'}$.
 Recall that $N\in M_0$.
 From IV. above and the definition of $A'$ it follows that $N<_\in W^{NA'}$.
 From $ad$ III. above it follows that there is no
 choice of elements with $N<_\in W^{NA'}$, a contradiction.

 The case $N>_\in A'$ is analogous.
\end{proof}
\begin{proof}
  Proof of Theorem \ref{univ2} follows by combining Lemma \ref{ext} and fact that extension property imply both universality and ultrahomogeneity of the partial order.
\end{proof}
\begin{example}
Consider partial order $(P,\leq_P)$ depicted in Figure \ref{poset}. The function
$c$ embedding  $(P,\leq_P)$ to $(\He,\leq_\He)$ can be defined as:
  \begin{eqnarray}
c(1)&=&\TA\nonumber, \\
c(2)&=&\TB\nonumber, \\
c(3)&=&\TC\nonumber, \\
c(4)&=&\TD\nonumber.
  \end{eqnarray}
\end{example}
\subsection{Remark on Conway's surreal numbers}
Recall the definition of surreal numbers, see \cite{conway}. (For a
recent generalization see \cite{ehrlich}). Surreal numbers are defined
recursively together with their linear order.  We briefly indicate how
the partial order $(\He,\leq_\He)$ fits into this scheme.
\begin{defn}
  A surreal number is a pair $x=\{x^L|x^R\}$, where every member of the sets $x^L$ and
  $x^R$ is a surreal number and every member of $x^L$ is strictly smaller than
  every member of $x^R$.
  
  We say that a surreal number $x$ is less than or equal to the surreal number $y$ if and
  only if $y$ is not less than or equal to any member of $x^L$ and any member of $y^R$ is not
  less than or equal to $x$.
  
  We will denote the class of surreal numbers by $\Sur$.
\end{defn}

$\He$ may be thought of as a subset of $\Sur$ (we recursively add $\apple$ to express pairs $x^L$, $x^R$).  The recursive definition of
$A\in\He$ leads to the following order which we define explicitly:

\begin{defn}
  For elements $A,B\in \He$ we write $A\leq_\Sur B$, when there is no $l\in A_L$ such that $B\leq_\Sur l$
  and no $r\in B_R$ such that $r\leq_\Sur A$.
\end{defn}

$\leq_\Sur$ is a linear order of $\He$ and it is the restriction of Conway's order.
It is in fact a linear extension of the partial order $(\He,\leq_\in)$:

\begin{thm}
  For any $A,B\in \He$, $A<_\in B$ implies $A<_\Sur B$.
\end{thm}
\begin{proof}
  We proceed by induction on $l(A)+l(B)$.

  For empty $A$ and $B$ the theorem holds as they are not comparable by $<_\in$.

  Let $A<_\in B$ with $W^{AB}$ as a witness.  If $W^{AB}\neq A,B$, then $A<_\Sur W^{AB} <_\Sur B$ by induction.  In the case $A \in B_L$, then $A<_\Sur B$ from the definition of $<_\Sur$.
\end{proof}

\section{Universality of Graph Homomorphisms}

Perhaps the most natural order between finite models is induced by homomorphisms.
The universality of the homomorphism order for the class of all finite
graphs was first shown by \cite{Pultr}.  

Numerous other classes followed (see e. g. \cite{Pultr}) but planar graphs
(and other topologically restricted classes) presented a problem.

The homomorphism order on the class of finite paths was studied in \cite{NZhu}.
It has been proved it is a dense partial order (with the exception of a few
gaps which were characterized; these gaps are formed by all core-path of height
$\leq 4$).  \cite{NZhu} also rises (seemingly too ambitious) question whether
it is a universal partial order. 
  This has
been resolved in \cite{HN-paths,HN-trees} by showing that finite oriented paths
with homomorphism order are universal. In this section we give a new proof of
this result.  The proof is simpler and yields a stronger result (see Theorem
\ref{klacky}).

Recall that
an {\em oriented path $P$} of length $n$ is any oriented graph $(V,E)$ where
$V=\{v_0,v_1,\ldots,v_n\}$ and for every $i=1,2,\ldots,n$ either
$(v_{i-1},v_i)\in E$ or $(v_i,v_{i-1})\in E$ (but not both), and there are no
other edges.  Thus an oriented path is any orientation of an undirected path.

Denote by $(\Paths,\leq_\Paths)$ the class of all finite paths ordered by homomorphism
order.
Given paths $P=(V,E)$, $P'=(V',E')$, a {\em homomorphism} is a mapping $\varphi:V\to V'$ which preserves edges:
$$(x,y)\in E\implies (\varphi(x),\varphi(y))\in E'.$$  For paths $P$ and $P'$ we write $P\leq_\Paths P'$ if and only if there is homomorphism $\varphi:P\to P'$

To show the universality of oriented paths, we will construct an embedding of
$(\Periodic,\subseteq)$ to $(\Paths,\leq_\Paths)$.  Recall that the class $\Periodic$ denotes the class of all periodic subsets of $\Z$ (see Section \ref{preiodickesection}).
This is a new feature, which gives
a new, more streamlined and shorter proof of the \cite{HN-paths}.
  The main difference of the proof in \cite{HN-paths,HN-trees} and the one presented here is the use of $(\Periodic,\subseteq)$ as the
base of the representation instead of $(\TV,\leq_\TV)$.  The linear nature of graph
homomorphisms among oriented paths make it very difficult to adapt many-to-one
mapping involved in $\leq_\TV$. The cyclic mappings of $(\Periodic,\subseteq)$ are easier
to use.

Let us introduce terms and notations that are useful when speaking of homomorphisms
between paths. (We follow standard notations as e.g. in \cite{HellNesetril, NZhu}.)

While oriented paths do not make a difference between initial and terminal
vertices, we will always consider paths in a specific order of vertices from
the initial to the terminal vertex.  We denote the initial vertex $v_0$ and the
terminal vertex $v_n$ of $P$ by $in(P)$ and $term(P)$ respectively.  For a path $P$ we will
denote by $\overleftarrow{P}$ the flipped path $P$ with order of vertices
$v_n,v_{n-1},\ldots,v_0$.  For paths $P$ and $P'$ we denote by $PP'$ the path
created by the concatenation of $P$ and $P'$ (i.e. the disjoint union of $P$ and $P'$
with $term(P)$ identified with $in(P')$).

The {\em length} of a path $P$ is the number of edges in $P$.  The {\em
algebraic length} of a path $P$ is the number of forwarding minus the
number of backwarding edges in $P$.  Thus the algebraic length of a
path could be negative.  The {\em level $l_P(v_i)$} of $v_i$ is the algebraic
length of the subpath $(p_0,p_1,\ldots,p_i)$ of $P$.  The {\em distance} between
vertices $p_i$ and $p_j$, $d_P(p_i,p_j)$, is given by $|j-i|$. The {\em
algebraic distance, $a_P(p_i,p_j)$}, is $l_P(v_j)-l_P(v_i)$.

Denote by $\varphi:P\to P'$ a homomorphism from path $P$ to $P'$.  Observe that
we always have $d_P(p_i,p_j)\leq d_{P'}(\varphi(p_i),\varphi(p_j))$ and
$a_P(p_i,p_j)=a_{P'}(\varphi(p_i),\varphi(p_j))$.  We will construct paths in
such a way that  every homomorphism $\varphi$ between path $P$ and $P'$ must
map the initial vertex of $P$ to the initial vertex of $P'$ and thus preserve levels of
vertices (see Lemma \ref{zacatel} bellow).

  \begin{figure}[t]
    \begin{center}
      \includegraphics[width=12cm]{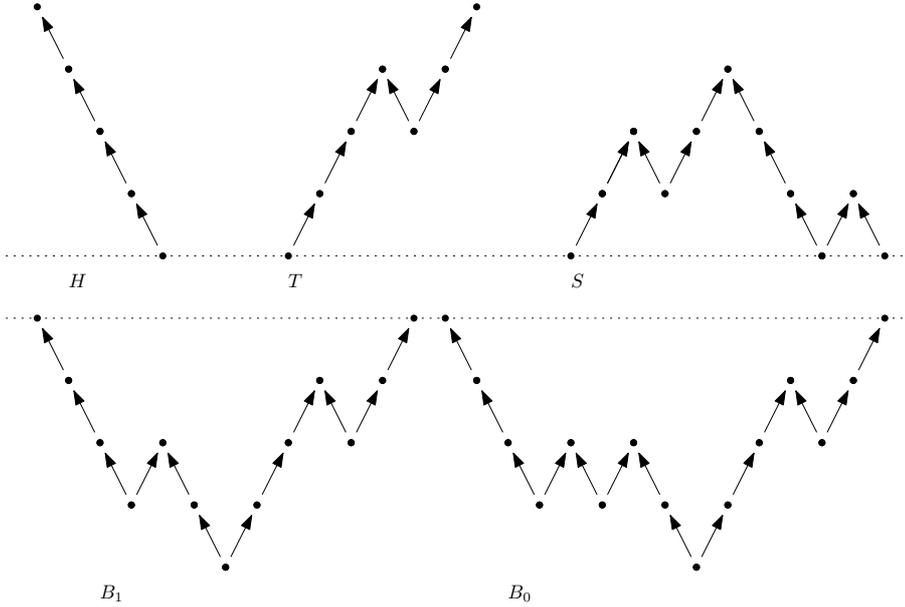}
    \end{center}
    \caption{Building blocks of $p(W)$.}
    \label{cesty}
  \end{figure}
The basic building blocks if our construction are the paths shown in Figure
\ref{cesty} ($H$ stands for {\em head}, $T$ for {\em tail}, $B$ for {\em body} and $S$ for {\em \v
sipka}---arrow in Czech language). Their initial vertices appear on the left,
terminal vertices on the right.  Except for $H$ and $T$ the paths are balanced
(i.e. their algebraic length is $0$). We will construct paths by concatenating
copies of these blocks. $H$ will always be the first path, $T$ always the last.  (The
dotted line in Figure \ref{cesty} and Figure \ref{p01110} determines vertices with level $-4$.)

\begin{defn}
Given a word $W$ on the alphabet $\{0,1\}$ of length $2^n$, we assign path $p(W)$ recursively as follows:

\begin{enumerate}
 \item $p(0)=B_0$.
 \item $p(1)=B_1$.
 \item $p(W)=p(W_1)S\overleftarrow {p(W_2)}$ where $W_1$ and $W_2$ are words of length $2^{n-1}$ such that $W=W_1W_2$. 
\end{enumerate}
Put $\overline{p}(W)=Hp(W)T$.
\end{defn}

\begin{example}
\label{exp1}
For a periodic set $S$, $s(4,S)=0110$, we construct $\overline{p}(s(4,S))$ in the following way:

$$p(0)=B_0,$$
$$p(1)=B_1,$$
$$p(01)=B_0S\overleftarrow{B_1},$$
$$p(10)=B_1S\overleftarrow{B_0},$$
$$p(0110)=B_0S\overleftarrow {B_1}SB_0\overleftarrow S\overleftarrow {B_1},$$
$$\overline{p}(0110)=HB_0S\overleftarrow {B_1}SB_0\overleftarrow S\overleftarrow {B_1}T.$$

  \begin{figure}[t]
    \begin{center}
      \includegraphics[width=13cm]{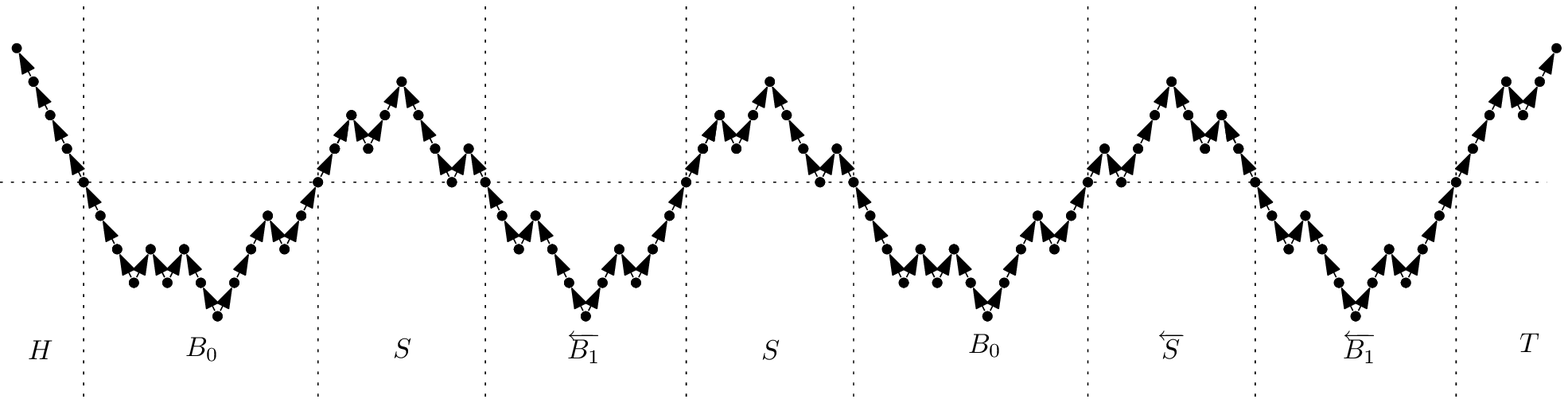}
    \end{center}
    \caption{$\overline{p}(0110)$.}
    \label{p01110}
  \end{figure}
See Figure \ref{p01110}.

\end{example}

The key result of our construction is given by the following:
\begin{prop}
\label{pscarkou}
Fix a periodic set $S$ of period $2^k$ and a periodic set $S'$ of period $2^{k'}$.
There is a homomorphism $$\varphi:\overline{p}(s(2^k,S))\to \overline{p}(s(2^{k'},S'))$$ if and only if $S\subseteq
S'$ and $k'\leq k$.

If a homomorphism $\varphi$ exists, then $\varphi$ maps the initial vertex of $\overline{p}(s(2^k,S))$ to the initial vertex of $\overline{p}(s(2^{k'},S'))$.
If $k'=k$ then $\varphi$ maps the terminal vertex of $\overline{p}(s(2^k,S))$ to the terminal vertex of $\overline{p}(s(2^{k'},S'))$. If $k'<k$
 then $\varphi$ maps the terminal vertex of $\overline{p}(s(2^k,S))$ to the initial vertex of $\overline{p}(s(2^{k'},S'))$.
\end{prop}

Prior to the proof of Proposition $\ref{pscarkou}$ we start with observations about homomorphisms between our special paths.
\begin{lem}
\label{zacatel}
Any homomorphism $\varphi:\overline{p}(W)\to \overline{p}(W')$ must map the initial vertex of $\overline{p}(W)$ to the initial vertex of $\overline{p}(W')$.
\end{lem}
\begin{proof}
$\overline{p}(W)$ starts with the monotone path of 7 edges. The homomorphism $\varphi$ must map this
path to a monotone path in $\overline{p}(W')$. The only such
subpath of $\overline{p}(W')$ is formed by first 8 vertices of $\overline{p}(W')$.

It is easy to see that $\varphi$ cannot flip the path:
If $\varphi$ maps the initial vertex of $\overline{p}(W)$ to the 8th vertex of $\overline{p}(W')$ then $\overline{p}(W)$
has vertices at level $-8$ and because homomorphisms must preserve algebraic
distances, they must map to the vertex of level $1$ in $\overline{p}(W')$ and there
is no such vertex in $\overline{p}(W')$.
\end{proof}
\begin{lem}
\label{phomo}
Fix words $W,W'$ of the same length $2^k$.  Let $\varphi$ be a homomorphism $\varphi:p(W)\to p(W')$.
 Then $\varphi$ maps the initial vertex of $p(W)$ to the initial vertex of $p(W')$
 if and only if $\varphi$ maps the terminal vertex of $p(W)$ to the terminal vertex of $p(W')$.
\end{lem}
\begin{proof}

We proceed by induction on length of $W$:

For $W=i$ and $W'=j$, $i,j\in\{0,1\}$ we have $p(W)=B_i$ and $p(W')=B_j$.
There is no homomorphism $B_1\to B_0$. The unique
homomorphism $B_0\to B_{1}$ has the desired properties. The only homomorphism $B_0\to B_0$ is the isomorphism $B_0\to B_0$.

In the induction step put $W=W_0W_1$ and $W'=W'_0W'_1$ where $W_0$, $W_1$, $W'_0$, $W'_1$ are words of length $2^{k-1}$. We have $p(W)=p(W_0)S\overleftarrow {p(W_1)}$ and $p(W')=p(W'_0)S\overleftarrow {p(W'_1)}$. 

First assume that $\varphi$ maps $in(p(W))$ to $in(p(W'))$.
Then $\varphi$ clearly maps $p(W_0)$ to $p(W'_0)$ and thus by the induction
hypothesis $\varphi$ maps $term(p(W_0))$ to $term(p(W'_0))$. Because the vertices of
$S$ are at different levels than the vertices of the final blocks $B_0$ or $B_1$ of
$p(W_0')$, a copy of $S$ that follows in $p(W)$ after $p(W_0)$ must map to a copy 
of $S$ that follows in $p(W')$ after $p(W'_0)$.
Further $\varphi$ cannot flip $S$ and thus $\varphi$ maps $term(S)$ to $term(S)$.
By same argument $\varphi$ maps $p(W_1)$ to $p(W'_1)$. The initial vertex of $p(W_1)$ is the terminal vertex of $p(W)$ and it must map
to the initial vertex of $p(W'_1)$ and thus also the terminal vertex of $p(W')$.

The second possibility is that $\varphi$ maps  $term(p(W))$ to $term(p(W'))$. This can be handled similarly (starting from the terminal vertex of paths
in the reverse order).
\end{proof}
\begin{lem}
\label{pscarkou2}
Fix periodic sets $S,S'$ of the same period $2^k$.
There is a homomorphism $$\varphi:p(s(2^k,S))\to p(s(2^k,S'))$$ mapping $in(p(s(2^k,S)))$ to $in(p(s(2^k,S')))$ if and only if $S\subseteq S'$. 
\end{lem}
\begin{proof}
If $S\subseteq S'$ then the Lemma follows from the construction of $p(s(2^k,S))$. Every digit 1 of $s(2^k,S)$ has a corresponding copy of $B_1$ in $p(s(2^k,S))$ and every
digit 0 has a corresponding copy of $B_0$ in $p(s(2^k,S))$.  It is easy to build a homomorphism $\varphi$ by concatenating a homomorphism $B_0\to B_1$
and identical maps of $S$, $B_0$ and $B_1$.

In the opposite direction, assume that there is a homomorphism $\varphi$ from $p(s(2^k,S))$ to $p(s(2^k,S'))$. By the assumption and Lemma \ref{phomo}, $\varphi$ must be map
$term(p(s(2^k,S)))$ to $term(p(s(2^k,S')))$. Because $S$ use vertices at different
levels than $B_0$ and $B_1$, all copies of $S$ must be mapped to copies of $S$. Similarly copies of $B_0$ and $B_1$ must be mapped to copies of $B_0$ or $B_1$. If $S\not \subseteq S'$ then there is position $i$ such that $i$-th letter of $s(2^k,S)$ is 1 and $i$-th letter of $s(2^k,S')$ is 0. It follows that the copy of $B_1$ corresponding to this letter would have to map to a copy of $B_0$. This contradicts with the fact that there is no homomorphism $B_1\to B_0$.
\end{proof}

\begin{lem}[folding]
\label{folding}
For a word $W$ of length $2^k$, there is a homomorphism $$\varphi:\overline{p}(WW)\to \overline{p}(W)$$ mapping $in(\overline{p}(WW))$ to $in(\overline{p}(W))$
and $term(\overline{p}(WW))$ to $in(\overline{p}(W))$.
\end{lem}
\begin{proof}
By definition $$\overline{p}(WW)=Hp(W)S\overleftarrow{p(W)}T$$ and $$\overline{p}(W)=Hp(W)T.$$ 
The homomorphism $\varphi$
maps the first copy of $p(W)$ in $\overline{p}(WW)$ to a copy of $p(W)$ in $\overline{p}(W)$, a copy of $S$
is mapped to $T$ such that the terminal vertex of $S$ maps to the initial vertex of $T$
and thus it is possible to map a copy of $\overleftarrow{p(W)}$ in $\overline{p}(WW)$ to the
same copy of $p(W)$ in $\overline{p}(W)$.
\end{proof}

We will use the folding Lemma iteratively. By composition of homomorphisms there is
also homomorphism $p(WWWW)\to p(WW)\to p(W)$. (From the path constructed from $2^k$ copies of $W$ to
$p(W)$.)
\begin{proof}[Proof (of Proposition \ref{pscarkou})]
Assume the existence of a homomorphism $\varphi$ as in Proposition \ref{pscarkou}.
First observe that $k'\leq k$ (if $k<k'$ then there is a copy of $T$ in $\overline{p}(s(2^k,S))$ would have to map
into the middle of $\overline{p}(s(2^{k'},S'))$, but there are no vertices at the level 0 in
$\overline{p}(s(2^{k'},S'))$ except for the initial and terminal vertex).

For $k=k'$ the statement follows directly from Lemma \ref{pscarkou2}.

For $k'<k$ denote by $W''$ the word that consist of $2^{k-k'}$ concatenations of $W'$. Consider a homomorphism $\varphi'$ from $p(W)$ to $p(W'')$ 
mapping $in(p(W)$ to $in(p(W''))$. $W$ and $W''$
have the same length and such a homomorphism exists by Lemma \ref{pscarkou2} if
and only if $S\subseteq S'$. Applying Lemma \ref{folding} there is a homomorphism $\varphi'': p(W'')\to p(W')$.
A homomorphism $\varphi$ can be obtained by composing $\varphi'$ and $\varphi''$.
It is easy to see that any homomorphism $\overline{p}(W)\to \overline{p}(W')$ must follow the same scheme of
``folding'' the longer path $\overline{p}(W)$ into $\overline{p}(W')$ and thus there is a homomorphism
$\varphi$ if and only if $S\subseteq S'$. We omit the details.
\end{proof}

For a periodic set $S$ denote by $S^{(i)}$ the inclusion maximal
periodic subset of $S$ with period $i$. (For example for $s(4,S)=0111$ we have $s(2,S^{(2)})=01$.)
\begin{defn}
For $S\in \Periodic$ let $i$ be the minimal integer such that $S$ has period $2^i$.
Let $\Embed{\Periodic}{\Paths}(S)$  be the concatenation of the paths $$H,$$
$$\overline{p}(s(1,S^{(1)}))\overleftarrow{\overline{p}(s(1,S^{(1)}))},$$
$$\overline{p}(s(2,S^{(2)}))\overleftarrow{\overline{p}(s(2,S^{(2)}))},$$
$$\overline{p}(s(4,S^{(4)}))\overleftarrow{\overline{p}(s(4,S^{(4)}))},$$
\centerline{\ldots,}
$$\overline{p}(s(2^{i-1},S^{(2^{i-1})}))\overleftarrow{\overline{p}(s(2^{i-1},S^{(2^{i-1})}))},$$
$$\overline{p}(s(2^i,S))\overleftarrow {\overline{p}(s(2^i,S))}.$$
\end{defn}

\begin{thm}
$\Embed{\Periodic}{\Paths}(v)$ is an embedding of $(\Periodic,\subseteq)$ to $(\Paths,\leq_\Paths)$.
\end{thm}
\begin{proof}
Fix $S$ and $S'$ in $\Periodic$ of periods $2^i$ and $2^{i'}$ respectively.

Assume that $S\subseteq S', i>i'$.  Then the homomorphism $\varphi:\Embed{\Periodic}{\Paths}(S)\to\Embed{\Periodic}{\Paths}(S')$ can be constructed
via the concatenation of homomorphisms:
$$H\to H,$$
$$\overline{p}(s(1,S^{(1)}))\overleftarrow{\overline{p}(s(1,S^{(1)}))}\to \overline{p}(s(1,S'^{(1)}))\overleftarrow{\overline{p}(s(1,S'^{(1)}))},$$
$$\overline{p}(s(1,S^{(2)}))\overleftarrow{\overline{p}(s(1,S^{(2)}))}\to \overline{p}(s(2,S'^{(2)}))\overleftarrow{\overline{p}(s(2,S'^{(2)}))},$$
\centerline{\ldots,}
$$\overline{p}(s(2^{i'-1},S^{(2^{i'-1})}))\overleftarrow{\overline{p}(s(2^{i'-1},S^{(2^{i'-1})}))}\to \overline{p}(s(2^{i'-1},S'^{(2^{i'-1})}))\overleftarrow{\overline{p}(s(2^{i'-1},S'^{(2^{i'-1})}))},$$
$$\overline{p}(s(2^{i'},S^{(2^{i'})}))\overleftarrow{\overline{p}(s(2^{i'},S^{(2^{i'})}))}\to \overline{p}(s(2^{i'},S')),$$
$$\overline{p}(s(2^{i'+1},S^{(2^{i'+1})}))\overleftarrow{\overline{p}(s(2^{i'+1},S^{(2^{i'+1})}))}\to \overline{p}(s(2^{i'},S')),$$
\centerline{\ldots,}
$$\overline{p}(s(2^{i},S))\overleftarrow{\overline{p}(s(2^{i'},S))}\to \overline{p}(s(2^{i'},S')).$$
Individual homomorphisms exists by Proposition \ref{pscarkou}.  For $i\leq i'$ the
construction is even easier.

In the opposite direction assume that there is a homomorphism
$\varphi:\Embed{\Periodic}{\Paths}(S)\to\Embed{\Periodic}{\Paths}(S')$.
$\Embed{\Periodic}{\Paths}(S)$ starts by two concatenations of $H$ and thus 
a long monotone path and using a same argument as in Lemma \ref{zacatel}, $\varphi$
must map the initial vertex of $\Embed{\Periodic}{\Paths}(S)$ to the initial vertex of
$\Embed{\Periodic}{\Paths}(S')$.  It follows that $\varphi$ preserves levels of vertices.
It follows that for every $k=1,2,4,\ldots, 2^i$, $\varphi$ must map $\overline{p}(s(k,S^{(k)})$ to $\overline{p}(s(k',S'^{(k')}))$  for some $k'\leq k, k'=1,2,4,\ldots, 2^{i'}$.  By application of
Proposition \ref{pscarkou} it follows that $S^{(k)}\subseteq S'^{(k')}$. In particular $S\subseteq S'^{(k')}$. This holds only if $S\subseteq S'$.
\end{proof}
\begin{thm}[\cite{HN-paths}]
\label{pathuniv}
The quasi order $(\Paths,\leq_\Paths)$ contains universal partial order.
\end{thm}

In fact our new proof of Corollary \ref{pathuniv} gives the following
strengthening for rooted homomorphisms of paths. A {\em plank} $(P,r)$ is an
oriented path rooted at the initial vertex $r=in(P)$.
Given planks $(P,r)$ and $(P',r')$, a homomorphism $\varphi:(P,r)\to (P',r')$ is a homomorphism
$P\to P'$ such that $\varphi(r)=r'$.

\begin{thm}
\label{klacky}
The quasi order formed by all planks ordered by the existence of homomorphisms contains a universal partial order.
\end{thm}

\section{Related results}
The universality of oriented paths implies the universality of the homomorphism order of many
naturally defined classes of structures (such as undirected planar or series-parallel graphs) ordered by homomorphism via the indicator
construction (see \cite{HN-trees}, \cite{NJared}).  By similar techniques the universality
of homomorphism the order on labelled partial orders is shown in \cite{lehtonen1}.

Lehtonen  and Ne\v set\v ril \cite{lehtonen2} consider also the partial order defined on boolean
functions in the following way.  Each clone $\mathcal C$ on a fixed base set $A$
determines a quasiorder on the set of all operations on $A$ by the following
rule: $f$ is a $\mathcal C$-minor of $g$ if $f$ can be obtained by substituting
operations from $\mathcal C$ for the variables of $g$.  Using embedding homomorphism order
on hypergraphs, it can be shown that a clone $\mathcal C$ on $\{0,1\}$ has the property
that the corresponding $\mathcal C$ minor partial order is universal if and only if $\mathcal C$
is one of the countably many clones of clique functions or the clone of
self-dual monotone functions (using the classification of Post classes).

It seems that in most cases the homomorphism order of classes of relational
structures is either universal or fails to be universal for very simple reasons
(such as the absence of infinite chains or anti-chains).
\cite{NJared} look for minimal minor closed classes of
graphs that are dense and universal.  They show that $(\Paths,\leq_\Paths)$ is
a unique minimal class of oriented graphs which is both universal and dense.
Moreover, they show a dichotomy result for any minor closed class $\K$ of
directed trees. $\K$ is either universal or it is well-quasi-ordered.  Situation
seems more difficult for the case of undirected graphs, where such minimal
classes are not known and only partial result on series-parallel graphs was
obtained.

\section{Acknowledgement}
We thank Andrew Goodall and Robert \v S\'amal for help and several remarks which improved the quality of this paper.
We thank to the anonymous referee for an excellent job.

\end{document}